\documentclass[]{article}
\usepackage[utf8]{inputenc} 
\usepackage[T1]{fontenc}    
\usepackage{hyperref}       
\usepackage{url}            
\usepackage{microtype}
\usepackage{graphicx}
\usepackage{array}
\usepackage{color} 
\usepackage{tabularx}
\usepackage{amsmath,mathdots,amsthm}
\usepackage{amssymb}
\usepackage{amsfonts}
\usepackage{txfonts}
\usepackage{arxiv}
\usepackage{xcolor}
\usepackage{bm}
\usepackage{soul}
\usepackage{float}
\usepackage{cite}
\usepackage{booktabs, multirow}

\newcommand{\halb}{\frac{1}{2}}
\hypersetup{
	colorlinks=true,                          
	linkcolor=blue, 
	citecolor=blue, 
	urlcolor=black  
} 
\graphicspath{{Figures/}}

\newtheorem{rmrk}{Remark}
\newtheorem*{rmrk*}{Remark}

\newtheorem*{definition*}{Definition}

\newtheorem*{thm*}{Theorem}

\newtheorem*{prp*}{Proposition}

\newcommand{\nrm}[1]{ \left\vert\left\vert #1 \right\vert\right\vert }

\renewcommand{\div}[1]{\mathrm{div}\left( #1 \right)}
\newcommand{\pd}[2]{\frac{\partial #1}{\partial #2}}
\newcommand{\pdt}[1]{\frac{\partial #1}{\partial t}}
\newcommand{\pdv}[2]{\frac{\delta #1}{\delta #2}}
\newcommand{\prn}[1]{\left( #1 \right)}

\newcommand{\q}{\mathbf{q}}
\newcommand{\p}{\mathbf{p}}

\newcommand{\grad}[1]{\nabla\left( #1 \right)}
\newcommand{\dx}{\Delta x}
\newcommand{\dy}{\Delta y}
\newcommand{\dt}{\Delta t}

\newcommand{\x}{\mathbf{x}}

\newcommand{\calD}{\mathcal{D}}
\newcommand{\calG}{\mathcal{G}}
\newcommand{\calF}{\mathcal{F}}
\newcommand{\calL}{\mathcal{L}}
\newcommand{\chii}{\raisebox{0.4ex}{$\chi$}}

\renewcommand{\j}{\mathbf{J}} 

\title{A first-order hyperbolic reformulation of the Cahn-Hilliard equation}

\author{
	\href{https://orcid.org/0000-0002-7150-3313}{\includegraphics[scale=0.06]{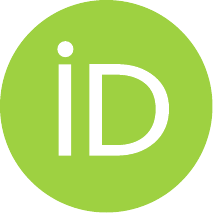}\hspace{1mm}Firas Dhaouadi}
	\thanks{corresponding author.\newline Email addresses: \texttt{firas.dhaouadi@unitn.it} (F. Dhaouadi),  \texttt{michael.dumbser@unitn.it} (M. Dumbser),
		\texttt{sergey.gavrilyuk@univ-amu.fr} (S. Gavrilyuk)} $^{\;a}$, 
	\hspace{3mm}
	\href{https://orcid.org/0000-0002-8201-8372}{\includegraphics[scale=0.06]{orcid.pdf}\hspace{1mm}Michael Dumbser$^a$}
	\hspace{3mm}
	\href{https://orcid.org/0000-0003-4605-8104}{\includegraphics[scale=0.06]{orcid.pdf}\hspace{1mm}Sergey Gavrilyuk$^b$} 
}

\affiliation{$^a$ Department of Civil, Environmental and Mechanical Engineering, University of Trento, Via Mesiano 77, 38123 Trento, Italy. \\ 
	$^b$Aix-Marseille University and CNRS UMR 7343 IUSTI, 5 rue Enrico Fermi, 13453 Marseille, France.
}
\date{}

\renewcommand{\headeright}{F. Dhaouadi, M. Dumbser, S. Gavrilyuk}
\renewcommand{\undertitle}{}
\renewcommand{\shorttitle}{}

\hypersetup{
	pdftitle={A first-order hyperbolic reformulation of the Cahn-Hilliard equation},
	pdfsubject={q-bio.NC, q-bio.QM},
	pdfauthor={Firas Dhaouadi, Sergey Gavrilyuk},
	pdfkeywords={First keyword, Second keyword, More},
}

\begin{document}
	\maketitle

	\begin{abstract}
		In this paper we present a new first-order hyperbolic reformulation of the Cahn-Hilliard equation. The model is obtained from the combination of augmented Lagrangian techniques proposed earlier by the  authors of this paper, with a classical Cattaneo-type relaxation that allows to reformulate diffusion equations as augmented first order hyperbolic systems with stiff relaxation source terms. The proposed system is proven to be hyperbolic and to admit a Lyapunov functional, in accordance with the original equations. A new numerical scheme is proposed to  solve the original Cahn-Hilliard equations based on conservative semi-implicit finite differences, while the hyperbolic system was numerically solved by means of a classical second order MUSCL-Hancock-type finite volume scheme. The proposed approach is validated through a set of classical benchmarks such as spinodal decomposition, Ostwald ripening and exact stationary solutions. 
	\end{abstract}
	
	\keywords{Cahn-Hilliard equation \and First order hyperbolic systems \and Relaxation models \and augmented Lagrangian approach \and Finite Volume schemes}
	
	\section{Introduction}
	The Cahn-Hilliard equation classically writes as
	\begin{equation}
		\pdt{c} = \Delta \prn{\frac{dg}{dc}-\gamma \Delta c},
		\label{eq:CH1}
	\end{equation}
where the variable $c$ is the order parameter indicating the phases. $\gamma$ is a small parameter such that $\sqrt{\gamma}$ is a characteristic length (thickness) of the transition zone from one phase to the other. The function $g(c)$ is a double-well potential whose exact expression will be provided later, and will be kept in a generic form. Equation \eqref{eq:CH1} was proposed in 1958 by John W. Cahn and John E. Hilliard in their seminal paper \cite{cahn1958free} as a model to describe nucleation in binary mixtures. The equation, in its different versions, has since then been used as a canonical model for two-phase flows \cite{gurtin1996two,anderson1998diffuse}, able to describe phase ordering \cite{bray2002} and phase separation processes such as spinodal decomposition \cite{cahn1961spinodal} and Ostwald ripening \cite{chakrabarti1993late}, with applications in different fields such as mechanics of living tissues \cite{wise2008three,agosti2017cahn}, coarsening in binary alloys \cite{cahn1961spinodal,dreyer2000study}, astronomy \cite{tremaine2003origin}, etc.
The Cahn-Hilliard equation \eqref{eq:CH1} can be cast into a conservation-law form which writes 
	\begin{equation}
		\pdt{c} + \div{\j} = 0,
		\label{eq:CH_J}
	\end{equation}
	where the mass flux $\j$ is assumed to obey a generalized Fick's law such that 
	\begin{equation*}
		\j = -\nabla \mu,
	\end{equation*}
	and $\mu$ is the chemical potential of the system given by
		\begin{equation*}
		\mu = \pdv{f}{c} = \pd{f}{c} - \div{\pd{f}{\nabla c}} = \frac{dg}{dc} - \gamma \Delta c,
	\end{equation*}
where
	\begin{equation*}
	f(c,\nabla c) = g(c) + \frac{\gamma}{2}\nrm{\nabla c}^2,
\end{equation*}
 is a double-well potential regularized by a gradient term. The latter was proposed by Cahn and Hilliard \cite{cahn1958free} and allows to stabilize the system in the spinodal region, \textit{i.e.} where the potential is non-convex, resulting in backward diffusion. In this context, the model bears similarities to Euler-Korteweg-Van der Waals systems \cite{vdwaals,van1979thermodynamic} for two-phase flows in the sense that both use a double-well potential, regularized by gradient terms of the order parameter. Further connections between both models were drawn in \cite{rohde2020homogenization}.
 Equation \eqref{eq:CH1} admits the Lyapunov functional
	\begin{equation*}
		F(c,\nabla c)=  \int_\calD f(c,\nabla c)   \  d D, 
	\end{equation*}
where $\cal D$ is a space domain.
Indeed, multiplying both sides of \eqref{eq:CH_J} by $\delta f/\delta c$, allows to obtain the dissipation law
	\begin{equation*}
		\pdt{f} + \div{\mu J} = -\nrm{\nabla \mu}^2,
	\end{equation*} 
	which in integral form writes 
	\begin{equation*}
		\frac{dF}{dt} =  -\int_\calD \nrm{\nabla \mu}^2   \  d D \leq 0.
	\end{equation*}
Several explicit forms of $g(c)$ are proposed in the literature, the most popular being
\begin{equation}
	g(c) = \frac{(c^2-1)^2}{4},
	\label{eq:g(c)}
\end{equation}
and which will be used later, in the numerical computations. Under this form, $c$ is expected to remain within the physically admissible region $[-1,1]$, where the limit values characterize pure phases (for example $c=-1$ for phase 1 and $c=1$ for phase 2). Note however that, as pointed out in the monograph \cite{miranville2019cahn} (precisely \textit{Remark 4.10}), the Cahn-Hilliard equation with the specific double-well potential \eqref{eq:g(c)} does not satisfy  the maximum principle. Initial data within the interval $[-1,1]$, can lead to solutions extending slightly beyond this range (see the counter-example given in \cite{pierre2011etude}).  This is why the simplified double-well potential $g(c)$  is often replaced by more general forms \cite{miranville2019cahn}. Nevertheless, the cubic approximation  $\displaystyle\frac{dg}{dc}= c^3-c$ remains of practical interest, and solutions usually remain within the physical range, since the equilibrium points $c=\pm 1$ are asymptotically stable. 
Well-posed problem formulations for the Cahn-Hilliard equation can be found, for example, in  \cite{miranville2019cahn,elliott1986cahn}.	

Several other variants of the Cahn-Hilliard equations have been proposed in the literature, out of which the most known is the singularly perturbed Cahn-Hilliard equation, also referred to as Cahn-Hilliard equation with inertial term and which writes \cite{galenko2001phase}
\begin{equation}
	\beta \pd{^2 c}{t^2} +   \pdt{c} = \Delta \prn{\frac{dg}{dc}-\gamma \Delta c}, \qquad 0 < \beta \ll 1.
	\label{eq:CH_inertia}
\end{equation}
The mathematical properties of \eqref{eq:CH_inertia} are studied, for example, in \cite{debussche1991singular,gatti2005hyperbolic,grasselli20092d,bonfoh2010singularly}. In particular in \cite{debussche1991singular}, the existence of Lyapunov functionals was established for a particular class of solutions. One of the main flaws of such a formulation is the loss of mass conservation. The Cahn-Hilliard equation can also be coupled with continuum motion  \cite{lowengrub1998quasi,mulet2024implicit}.

From the numerical point of view, the Cahn-Hilliard equation presents several difficulties to account for, which spurred numerous contributions in the last decades, see for instance \cite{elliott1987numerical,du1991numerical,feng2004error,shen2010numerical,antonietti2016c, chen2016convergence} . First, since it is fourth order in space, an explicit scheme would require extremely small time steps that scale with the fourth power of the mesh size \cite{rogers1988numerical,vollmayr2003fast}, thereby favoring implicit strategies in a natural manner. Second, the change of convexity in the double-well potential also requires careful treatment and a typical way to address it is a convex-nonconvex splitting, see for example \cite{eyre1998unconditionally, vollmayr2003fast, lee2024operator, mulet2024implicit}. 

In this paper, we propose a new first order hyperbolic reformulation of the Cahn-Hilliard equation \eqref{eq:CH1}, by means of an augmented first-order hyperbolic relaxation system of equations for the following reasons. First, given the numerical challenges to overcome for solving Equation \eqref{eq:CH1}, it seems of practical interest to provide an alternative formulation that can be numerically solved in a straightforward manner by using standard techniques for hyperbolic relaxation systems without much additional efforts and, in particular, without prohibitively small time step sizes. Second, the Cahn-Hilliard Equation \eqref{eq:CH1} violates the principle of causality, similarly to the heat equation, due to the Laplace operators, whereas a hyperbolic system provides an alternative and causal description based on bounded signal speeds. 

The paper is organized as follows. In section \ref{sec:hyperbolic}, the relaxation procedure allowing to obtain a first-order hyperbolic approximation of the Cahn-Hilliard equation \eqref{eq:CH1} is detailed and the obtained system is analyzed. The existence of a Lyapunov functional for the latter is demonstrated and the hyperbolicity is proven. In section \ref{sec:numeric}, a numerical scheme based on implicit fourth order conservative finite-differences is proposed to solve Equation \eqref{eq:CH1} in order to provide reliable reference solutions. Then, a finite-volume (FV) scheme based on the well-known second order accurate MUSCL-Hancock method for solving the proposed hyperbolic approximation is recalled. In section \ref{sec:results}, numerical results for the hyperbolic approximation of \eqref{eq:CH1} are provided and thoroughly analyzed and compared against available exact or numerical reference solutions. In particular, the effects of the relaxation parameters are carefully investigated via these numerical results. A summary of the main contributions and the possible future extensions of this work are given in section \ref{sec:conclusion}.

	\section{Hyperbolic approximation of the Cahn-Hilliard equation}
	\label{sec:hyperbolic}
	\subsection{Second-order approximation}
	The objective of this work is to formulate a first-order hyperbolic system approximating the fourth-order PDE \eqref{eq:CH1}. Note that in the latter, the fourth order term is only due to the fact that $f$ explicitly depends of $\nabla c$, resulting in a Laplace operator appearing in the chemical potential $\mu$. Therefore, the first step consists in separating the gradient term and the concentration $c$ in the free energy, which would allow us reduce the chemical potential $\mu$ to an algebraic expression of $c$. To this end, let us introduce the following action functional, where an additional integration over a finite time interval $\cal I$ is introduced :
	\begin{equation*}
		a = \int_{\cal I}\int_\calD   \calL \ d D dt, \quad \calL\prn{c,\varphi,\nabla \varphi,\pdt{\varphi}} = -g(c) - \frac{\gamma}{2}\nrm{\nabla\varphi}^2 - \frac{\alpha}{2}(c-\varphi)^2 + \frac{\beta}{2} \prn{\pdt{\varphi}}^2.
	\end{equation*}
	In this expression, $\varphi$ is a new variable substituting $c$ as the order parameter distinguishing the phases. As a matter of fact, $\nabla c$ is replaced by $\nabla \varphi$. In order to ensure consistency with the original system, a penalty term is introduced enforcing that $(c-\varphi)$ vanishes in the limit $\alpha\rightarrow+\infty$. The last term is introduced as an additional regularization of the action functional, with a small parameter $\beta$ . These steps and considerations together form what we call an augmented Lagrangian approach, which so far was applied to build hyperbolic approximations only to dispersive systems of equations, see for example \cite{favrie2017rapid,dhaouadi2019extended,bourgeois2020dynamics,dhaouadi2020augmented,busto2021high,dhaouadi2022films,dhaouadi2022NSK,gavrilyuk2022hyperbolic,gavrilyuk2024conduit}. Now considering the action $a$ and the corresponding Lagrangian $\calL$ we proceed as follows. Likewise to the original Cahn-Hilliard equation, we write the generalized Fick's law, which now becomes
	\begin{equation}
		\pdt{c} -\Delta \mu = 0, \quad \mu = -\pdv{\calL}{c} = -\pd{\calL}{c} = \frac{dg}{dc} + \alpha(c-\varphi),
		\label{eq:eq_with_pot}
	\end{equation}
	and which is thus completely independent of high-order terms. The time-evolution of the variable $\varphi$ is governed by the Euler-Lagrange equation corresponding to the variations of the action functional $a$ with respect to the latter and which writes
	\begin{equation*}
		\pdt{} \prn{\pd{\calL}{\varphi_t}}+ \div{\pd{\calL}{\nabla\varphi}} = \pd{\calL}{\varphi}.
	\end{equation*}
	Thus, one obtains the following equations
	\begin{subequations}
		\begin{align}
			&\pdt{c} - \div{\nabla\prn{\frac{dg}{dc} + \alpha(c-\varphi)}} = 0, \label{eq:2ndOrderCH_1}\\
			\beta &\pd{^2\varphi}{t^2}- \div{\gamma \nabla \varphi} = \alpha(c-\varphi),
			\label{eq:2ndOrderCH_2}
		\end{align}
		\label{eq:2ndOrderCH}
	\end{subequations}
	which form a system of second-order PDEs, approximating the Cahn-Hilliard equation \eqref{eq:CH1}. One can further highlight the connection to Equation \eqref{eq:CH1} by inserting the expression of $\alpha(c-\varphi)$ from \eqref{eq:2ndOrderCH_2} into \eqref{eq:2ndOrderCH_1} in order to obtain 
	\begin{equation*}
		\pdt{c} = \Delta\prn{\frac{dg}{dc} - \div{\gamma \nabla \varphi} + \beta \pd{^2\varphi}{t^2}},
	\end{equation*} 
	which is consistent with Equation \eqref{eq:CH1} for $\beta\rightarrow0$ and $\varphi\rightarrow c$.
	\subsection{First-order hyperbolic approximation}
	Starting from the second order approximation \eqref{eq:2ndOrderCH}, the aim now is to attain a first-order hyperbolic structure. We begin by addressing Equation \eqref{eq:2ndOrderCH_2} since it simply is a linear wave equation in $\varphi$ with a non-zero right-hand side. In this case we only need a proper change of variables that casts the latter into an augmented first-order system. We follow the idea proposed in  \cite{dhaouadi2019extended,dhaouadi2022films} and we promote the time derivative and the gradient of $\varphi$ to new independent variables $w$ and $\p$, respectively,  so that
	\begin{subequations}
		\begin{align}
			&w = \beta \pdt{\varphi}, \label{eq:w} \\ 
			&\p = \nabla \varphi. \label{eq:def_p}
		\end{align}
	\end{subequations}
	This new set of variables allows to recast Equation \eqref{eq:2ndOrderCH_2} into first order form, but it also provides a time-evolution equation for $\varphi$ that is Equation \eqref{eq:w}. Besides, one can also recover a time-evolution equation for $\p$ by differentiating Equation \eqref{eq:p} with respect to time and using \eqref{eq:w} to obtain 
	\begin{equation}
		\pdt{\p} - \frac{1}{\beta} \nabla w = 0.
		\label{eq:p}
	\end{equation}
	Thus, under these notations, one can convert Equation \eqref{eq:2ndOrderCH_2} into the following augmented first-order system 
		\begin{align*}
			& \pdt{w}  - \div{\gamma\p}  =-\alpha(\varphi-c), \\
			& \pdt{\p} -\frac{1}{\beta}\nabla w = 0, \\
			& \pdt{\varphi} = \frac{1}{\beta}w.
		\end{align*} 
	\begin{rmrk}
		For this change of variables to be truly consistent with Equation \eqref{eq:2ndOrderCH_2}, it is necessary that initial/boundary conditions be assigned properly. Indeed, Equation \eqref{eq:p} by itself cannot substitute the definition \eqref{eq:def_p}. One needs at least to impose also 
		\begin{equation*}
			\p(\x,t=0) = \nabla\varphi(\x,t=0), \quad \p(\x,t)\vert_{\partial \calD} = \nabla\varphi(\x,t)\vert_{\partial \calD},
		\end{equation*}
		which otherwise, may result in significant discrepancies in the solutions. 
	\end{rmrk}
	\begin{rmrk}
		By virtue of Equation \eqref{eq:p} and taking into account the previous remark, $\p$ satisfies a curl-involution constraint 
		\begin{equation*}
			\nabla \times \p = 0, \quad \forall t\geq0. 	
		\end{equation*} 
	\end{rmrk}
	It remains to address the PDE \eqref{eq:2ndOrderCH_1}. We propose here 
	a classical Maxwell-Cattaneo-type relaxation approach. We start from \eqref{eq:eq_with_pot}, and we introduce an independent vector variable $\q$ and a constant relaxation time $\tau$ such that $\q/\tau$ is the new mass flux. A second equation allows to relax $\q/\tau$ towards $-\nabla\mu$. The system is formulated as follows
		\begin{align*}
			& \pdt{c} + \div{\frac{1}{\tau}\q} = 0,\\
			& \pdt{\q}+ \nabla \mu  = -\frac{1}{\tau}\q.
		\end{align*} 
	Regrouping the whole first-order approximation of the Cahn-Hilliard equation allows to obtain the following final first-order system
	\begin{subequations}
		\begin{align}
			& \pdt{c} + \div{\frac{1}{\tau}\q} = 0,\\
			& \pdt{\q}+ \grad{\frac{dg}{dc}+\alpha(c-\varphi)}  = -\frac{1}{\tau}\q, \\
			& \pdt{w}  - \div{\gamma\p}  =-\alpha(\varphi-c), \\
			& \pdt{\p} -\frac{1}{\beta}\nabla w = 0, \\
			& \pdt{\varphi} = \frac{1}{\beta}w.
		\end{align} 
		\label{eq:CH_Hyp}
	\end{subequations}
	In what follows, we will refer to this system as the first-order hyperbolic reformulation of the original Cahn-Hilliard equation for the case when $\tau \to 0$, $\beta \to 0$ and $\alpha \to \infty$. 
	\begin{rmrk}
	Obviously, instead of using $\q$ as a main variable, one can use the variable $\q' = \q/\tau$ in order to simplify the flux and the relaxation source term. However, this will introduce a singular matrix ahead of time derivatives. 
	\end{rmrk}
	\begin{rmrk} 
	For the following scaling of the relaxation parameters
	\begin{equation*}
			\alpha = \gamma^{-1}, \quad \tau = \gamma^2, \quad \beta = \gamma^2,
	\end{equation*}
	we can prove via formal asymptotic expansions that the augmented system of equations \eqref{eq:CH_Hyp} is formally asymptotically consistent with the original Cahn-Hilliard equation \eqref{eq:CH1} for $\gamma \ll 1$. Indeed, similarly to what is done in \cite{hitz2020parabolic,dhaouadi2024eulerian},  if we assume that the following expansions hold 
	\begin{alignat*}{4}
	&c = c_0 &&+ \gamma \, c_1  &&+ \gamma^2\, c_2 &&+ O(\gamma^3), \\
	&\varphi = \varphi_0 &&+\gamma\, \varphi_1 &&+\gamma^2\, \varphi_2&&+ O(\gamma^3), \\
	&\q = \q_0 &&+ \gamma\, \q_1 &&+ \gamma^2\, \q_2 &&+ O(\gamma^3), \\
	&w = w_0 &&+ \gamma\, w_1 &&+ \gamma^2\, w_2 &&+ O(\gamma^3), \\
	&\p = \p_0 &&+\gamma\, \p_1 &&+\gamma^2\, \p_2 &&+ O(\gamma^3),
	\end{alignat*}
  then, by inserting these expansions into equations \eqref{eq:CH_Hyp}, sorting out the leading order terms of the powers of $\gamma$, one can obtain 
		\begin{equation*}
			\pdt{c_0} = \Delta\prn{ \frac{dg}{dc}(c_0) - \gamma\Delta c_0  } + O(\gamma^2),
		\end{equation*}
	\end{rmrk}
which is equivalent to Equation \eqref{eq:CH1} for $\gamma \ll 1$. 
	\subsection{Lyapunov functional}
	\begin{prp*}
		The system of equations \eqref{eq:CH_Hyp} admits the following Lyapunov functional
		\begin{equation*}
			E = \int_\calD  e(c,\varphi,\q,\p,w)  \ d D, \quad  e(c,\varphi,\p,w) = g(c) + \frac{\gamma}{2}\nrm{\p}^2 + \frac{\alpha}{2}(c-\varphi)^2 + \frac{1}{2\beta} w^2 + \frac{1}{2\tau} \nrm{\q}^2.
		\end{equation*}
	\end{prp*}
	\begin{proof}
		Let us express the fluxes in system \eqref{eq:CH_Hyp} as functions of the derivatives of $e$ with respect to the conserved variables as follows
		\begin{subequations}
			\begin{align*}
				& \pdt{c} + \div{\pd{e}{\q}} = 0,\\
				& \pdt{\q}+ \grad{\pd{e}{c}}  = -\pd{e}{\q}, \\
				& \pdt{w} - \div{\pd{e}{\p}}  = -\pd{e}{\varphi}, \\
				& \pdt{\p} -\nabla \prn{\pd{e}{w}} = 0, \\
				& \pdt{\varphi} = \pd{e}{w}.
			\end{align*} 
		\end{subequations}
		Summing up these equations, each multiplied by the corresponding derivative of $E$ allows to write
		\begin{align*}
			&  \pd{e}{c}\pdt{c} +\pd{e}{\q} \cdot \pdt{\q} +\pd{e}{w}\pdt{w} +\pd{e}{\p} \cdot\pdt{\p} + \pd{e}{\varphi} \pdt{\varphi}  \\
			+&\pd{e}{c}\div{\pd{e}{\q}} + \pd{e}{\q} \cdot \grad{\pd{e}{c}} -\pd{e}{\p}\cdot \nabla \prn{\pd{e}{w}} - \pd{e}{w}\div{\pd{e}{\p}} \\
			+&\pd{e}{\q} \cdot \pd{e}{\q} 
			+\pd{e}{w} \pd{e}{\varphi}    - \pd{e}{w}\pd{e}{\varphi} = 0,
		\end{align*} 
		which eventually leads to the equation
		\begin{equation}
			\pdt{e} + \div{\pd{e}{c} \pd{e}{\q} - \pd{e}{\p} \pd{e}{w}} = -\nrm{\pd{e}{\q}}^2 \leq 0,
			\label{eq:Edecay}
		\end{equation}
		or, assuming appropriate boundary conditions, in integral form 
		\begin{equation}
			\frac{dE}{dt} = - \int_\calD\ \nrm{\pd{e}{\q}}^2 \ d D \ \leq 0.
			\label{eq:Edecay_int}
		\end{equation}
		Hence, $E$ is a Lyapunov functional for \eqref{eq:CH_Hyp}.
	\end{proof}
	A final remark on the energy conservation equation \eqref{eq:Edecay}, is the fact that $\partial e/\partial\varphi$ does not appear in the evolution equation. However, it does play a role in the energetic exchanges within the system. Indeed, if one separates the energy density as follows 
	\begin{equation*}
		e(c,\varphi,\q,\p,w) = e_I(c,\varphi,\q) + e_{I\!I}(\p,w), \quad \textnormal{where} \quad  \begin{cases}
		\displaystyle	e_I = g(c) + \frac{\alpha}{2}(c-\varphi)^2 + \frac{1}{2\tau} \nrm{\q}^2, \\[0.5em]
		\displaystyle	e_{I\!I} =  \frac{\gamma}{2}\nrm{\p}^2 + \frac{1}{2\beta} w^2, 
		\end{cases}
	\end{equation*}
	one can obtain the following evolution equation for each of the energy parts
	\begin{align*}
			& \pdt{e_I} + \div{\pd{e_I}{c} \pd{e_I}{\q}} = -\pd{e_I}{\varphi} \pd{e_{I\!I}}{w} 	-\nrm{\pd{e_I}{\q}}^2, \\
			& \pdt{	e_{I\!I}} - \div{\pd{	e_{I\!I}}{\p} \pd{	e_{I\!I}}{w}} = \pd{e_I}{\varphi} \pd{e_{I\!I}}{w} ,	
	\end{align*}
	where one can see that $\partial e/\partial\varphi$ is the main term governing the energetic exchange between $e_I$ and $e_{I\!I}$.
	\subsection{Hyperbolicity}
	The system of equations \eqref{eq:CH_Hyp} is invariant by rotations of the $\mathrm{SO}_3$ group. This makes it sufficient to study its hyperbolicity  in the one-dimensional case, \textit{i.e} all state variables are assumed only to depend on one arbitrary dimension of space $x$ and on time $t$. Under these considerations, one can cast \eqref{eq:CH_Hyp} into quasilinear form
	\begin{equation*}
		\pdt{\mathbf{Q}} + \mathbf{A}( \mathbf{Q}) \pd{\mathbf{Q}}{x} = \mathbf{S} ( \mathbf{Q}),
	\end{equation*}
	where $\mathbf{Q}$ is the vector of conserved variables and $\mathbf{A}(\mathbf{Q})$ is the Jacobian matrix of the system, both given by
	\begin{align*}
		\mathbf{A}(\mathbf{Q}) = \left( 
		\begin{array}{@{}ccccc|c@{}}
			0 & \frac{1}{\tau} & 0 & 0 & 0 & \multicolumn{1}{c}{\multirow{4.5}{11pt}{$\mathbf{O}_{4,4}$}} \\
			\alpha + g''(c) & 0 & 0 & 0 & -\alpha & \\
			0 & 0 & 0 & -\gamma & 0 & \\
			0 & 0 & -\frac{1}{\beta} & 0 & 0 & \\
			\cmidrule[0.4pt]{1-6}
			\multicolumn{5}{c|}{\multirow{1}{*}{$\mathbf{O}_{5,5}$}} & \multirow{1}{*}{$\mathbf{O}_{5,4}$} 
		\end{array} 
		\right), \quad \mathbf{Q} = \prn{c,q_1,w,p_1,\varphi,q_2,q_3,p_2,p_3}^T,
	\end{align*}
	where $\mathbf{O}_{n,m}$ are $n\times m$ null matrices and $\displaystyle g''(c) = \frac{d^2g}{dc^2}$. Straightforward linear algebra shows that $\mathbf{A}(\mathbf{Q}) $ admits $9$ eigenvalues whose expressions are given hereafter 
	\begin{align*}
		&\lambda_{1} = -\frac{\sqrt{g''(c)+\alpha}}{\sqrt{\tau }}, \quad \lambda_{2} = -\frac{\sqrt{\gamma }}{\sqrt{\beta }}, \quad \lambda_{3-7} = 0, \quad  \lambda_{8} = \frac{\sqrt{\gamma }}{\sqrt{\beta }}, \quad \lambda_{9} = \frac{\sqrt{ g''(c)+\alpha}}{\sqrt{\tau }}.
	\end{align*}
	If we denote 
	\begin{equation*}
		\alpha_c = \left\vert \min_{c\in[-1,1]} (g''(c))\right\vert,
	\end{equation*}
	then, for $\alpha\geq\alpha_c$ all the eigenvalues are always real for $c\in[-1,1]$. Note that for the double-well potential \eqref{eq:g(c)}, $\alpha_c = 1$.  
	A full basis of linearly independent right eigenvectors can also be easily computed. These eigenvectors are collected below as the columns of the following matrix, in the same order as their corresponding eigenvalues
	\begin{equation*}
		\mathbf{R} = \left(
		\begin{array}{ccccccccc}
			-\frac{1}{\sqrt{\tau } \sqrt{\alpha + g''(c)}} & 0 & 0 & 0 & 0 & 0 & \frac{\alpha }{\alpha + g''(c)} & 0 & \frac{1}{\sqrt{\tau } \sqrt{\alpha + g''(c)}} \\
			1 & 0 & 0 & 0 & 0 & 0 & 0 & 0 & 1 \\
			0 & -\sqrt{\beta } \sqrt{\gamma } & 0 & 0 & 0 & 0 & 0 & \sqrt{\beta } \sqrt{\gamma } & 0 \\
			0 & 1 & 0 & 0 & 0 & 0 & 0 & 1 & 0 \\
			0 & 0 & 0 & 0 & 0 & 0 & 1 & 0 & 0 \\
			0 & 0 & 0 & 0 & 0 & 1 & 0 & 0 & 0 \\
			0 & 0 & 0 & 0 & 1 & 0 & 0 & 0 & 0 \\
			0 & 0 & 0 & 1 & 0 & 0 & 0 & 0 & 0 \\
			0 & 0 & 1 & 0 & 0 & 0 & 0 & 0 & 0 \\
		\end{array}
		\right).
	\end{equation*}
	One can check that these vectors are always linearly independent since
	\begin{equation*}
		\det{\mathbf{R}} = -4 \sqrt{\frac{\beta \gamma/ \tau}{\alpha +g''(c)} } \ne 0 \quad \forall \ \beta>0, \ \gamma >0, \ \tau >0, \ \alpha > \alpha_c.
	\end{equation*}
	Up to now, we did not use a specific form of $g(c)$, as to keep the analysis general. In what follows, we shall now use $g(c)$ given by \eqref{eq:g(c)} for its simplicity and for the fact that it allows to find exact solutions in a closed form ($tanh$-type phase transition fronts and periodic solutions, see subsection \ref{exact}).
	
	\section{Numerical Schemes}
	\label{sec:numeric}
	\subsection{Notations}
	In this section we shall use the following notations, presented in two space dimensions for the sake of simplicity. The reduction to one dimension is simply obtained by omitting the $y-$direction.
    The computational domain is $\Omega_c=[x_L,x_R]\times[y_L,y_R]$ discretized over $N_x\times N_y$ equally spaced Cartesian grid cells $\Omega_{i,j} = [x_{i-\halb}, x_{i+\halb}] \times [y_{j-\halb}, y_{j+\halb}]$, each of size $|\Omega_{i,j}| = \Delta x \times \Delta y$ where $\Delta x = (x_R-x_L)/N_x$ and $\Delta y = (y_R-y_L)/N_y$. Each computational cell $\Omega_{i,j}$ is uniquely identified by the coordinates of its center $(x_i,y_j),\ \forall (i,j)\in[1..N_x]\times[1..N_y]$ such that 
	\begin{equation*}
		x_i = x_L + \frac{i-1}{2} \Delta x \quad \forall i \in[1..N_x], \qquad y_j = y_L + \frac{i-1}{2} \Delta y \quad \forall i \in[1..N_y],
	\end{equation*}
	with the boundaries of the cells $\Omega_{i,j}$ given by $x_{i+\halb} = \halb (x_i + x_{i+1})$ and $y_{j+\halb} = \halb (y_j + y_{j+1})$, respectively. 
	The time is discretized such that $t^{n+1} = t^{n} + \Delta t$, where $t^0=0$ and $\Delta t$ is either fixed or determined from the CFL stability condition, depending on the numerical method used. Under these definitions, for any discrete quantity $\psi$, we shall use the notation $\psi^n_{i,j} = \psi(x_i,y_j,t^n)$.

	\subsection{Solving the original Cahn-Hilliard equation via a conservative semi-implicit finite difference scheme}
	We propose here a conservative semi-implicit finite differences scheme in order to solve  Equation \eqref{eq:CH1} numerically in a direct manner, without any hyperbolic approximation. This will serve later as reference solution for the discretization of the hyperbolic reformulation \eqref{eq:CH_Hyp}. We rewrite Equation \eqref{eq:CH1} as follows 
	\begin{equation*}
		\pdt{c} - \div{\mathbf{F}} + \gamma\Delta^2  c = 0,
	\end{equation*}
	where $\mathbf{F}$ is the flux given by 
	\begin{equation*}
		\mathbf{F} = \chii(c) \, \nabla c,  \quad \chii(c) = 3c^2-1.
	\end{equation*}
	Then, we propose a finite-volume scheme for the conservative part of the equation combined with finite differences for the bi-Laplacian operator so that the scheme writes as  
	\begin{align}
		c_{i,j}^{n+1} =& c_{i,j}^{n} + \frac{\dt}{\dx} \prn{\calF^{n+1}_{i+\frac{1}{2},j} - \calF^{n+1}_{i-\frac{1}{2},j}} + \frac{\dt}{\dy} \prn{\calG^{n+1}_{i,j+\frac{1}{2}} - \calG^{n+1}_{i,j-\frac{1}{2}}} - \gamma \Delta t \Delta^2_h c^{n+1}_{i,j}.
		\label{eq:scheme_org}
	\end{align}
	The intercell fluxes $\calF^{n+1}_{i+\frac{1}{2},j}$ and $\calG^{n+1}_{i,j+\frac{1}{2}}$, in the $x$ and $y$ directions respectively, are computed using semi-implicit fourth order accurate conservative finite-differences \cite{shuosher1,shuosher2} as follows 
	\begin{equation*}
		\calF^{n+1}_{i+\frac{1}{2},j} = \chii^{n}_{i+\frac{1}{2},j} \, \prn{\nabla_x c}^{n+1}_{i+\frac{1}{2},j}, \quad \text{where} \quad 
		\begin{cases}
			\displaystyle \chii^{n}_{i+\frac{1}{2},j} \simeq \frac{1}{12}\prn{7\, \chii^{n}_{i,j} - \chii^{n}_{i-1,j} +7\, \chii^{n}_{i+1,j}- \chii^{n}_{i+2,j}}, \\[0.5em]
			\displaystyle \prn{\nabla_x c}^{n+1}_{i+\frac{1}{2},j} \simeq -\frac{1}{12\,\dx}\prn{15\, c^{n+1}_{i,j} -15\, c^{n+1}_{i+1,j} + c^{n+1}_{i+2,j} - c^{n+1}_{i-1,j}},
		\end{cases}
	\end{equation*}
	\begin{equation*}
		\calG^{n+1}_{i+\frac{1}{2},j} = \chii^{n}_{i,j+\frac{1}{2}} \, \prn{\nabla_y c}^{n+1}_{i,j+\frac{1}{2}}, \quad \text{where} \quad 
		\begin{cases}
			\displaystyle \chii^{n}_{i,j+\frac{1}{2}} \simeq \frac{1}{12}\prn{7\, \chii^{n}_{i,j} - \chii^{n}_{i,j-1} +7\, \chii^{n}_{i,j+1}- \chii^{n}_{i,j+2}}, \\[0.5em]
			\displaystyle \prn{\nabla_y c}^{n+1}_{i,j+\frac{1}{2}} \simeq -\frac{1}{12\,\dy}\prn{15\, c^{n+1}_{i,j} -15\, c^{n+1}_{i,j+1} + c^{n+1}_{i,j+2} - c^{n+1}_{i,j-1}},
		\end{cases}
	\end{equation*}
	and where $\Delta^2_h c^{n+1}_{i,j}$ is a discretization of the bi-Laplacian operator in the cell-centers as follows
	\begin{align*}
		\Delta^2_h c^{n+1}_{i,j} =&  -\frac{1}{\dx^4} \prn{c_{i-2,j}^{n+1} -4 c_{i-1,j}^{n+1} + 6 c_{i,j}^{n+1} - 4 c_{i+1,j}^{n+1} + c_{i+2,j}^{n+1}} -\frac{1}{\dy^4} \prn{c_{i,j-2}^{n+1} -4 c_{i,j-1}^{n+1} + 6 c_{i,j}^{n+1} - 4 c_{i,j+1}^{n+1} + c_{i,j+2}^{n+1}} \nonumber \\
		& -   \frac{2}{\Delta x^2 \Delta y^2} \prn{c_{i-1,j-1}^{n+1} - 2c_{i,j-1}^{n+1} + c_{i+1,j-1}^{n+1} - 2c_{i-1,j}^{n+1} + 4c_{i,j}^{n+1} - 2c_{i+1,j}^{n+1} + c_{i-1,j+1}^{n+1} - 2c_{i,j+1}^{n+1} + c_{i+1,j+1}^{n+1}}.
	\end{align*}
	The use of fourth order conservative finite differences \cite{shuosher1,shuosher2} for the computation of the numerical fluxes $\calF^{n+1}_{i+\frac{1}{2},j}$ and $\calG^{n+1}_{i+\frac{1}{2},j}$ is necessary in order not to spoil the discretization of the bi-Laplacian, which contains fourth order spatial derivatives. Otherwise, for a simple second-order approximation of the numerical fluxes $\calF^{n+1}_{i+\frac{1}{2},j}$ and $\calG^{n+1}_{i+\frac{1}{2},j}$, the effective value of $\gamma$ would change in the modified equation of the scheme and it would not decrease with the mesh size, thus leading to zeroth order consistency errors.  
	Equation \eqref{eq:scheme_org}, which leads to a linear system for the $c_{i,j}^{n+1}$, is finally solved over the whole domain at each time step using a classical GMRES algorithm \cite{gmres}.
	
	\subsection{Solving the first order hyperbolic reformulation}
	The system of equations \eqref{eq:CH_Hyp} can be cast in conservative form in two dimensions of space such that 
	\begin{equation*}
		\pdt{\mathbf{Q}} + \pd{\mathbf{F}(\mathbf{Q})}{x} + \pd{\mathbf{G}(\mathbf{Q})}{y} = \mathbf{S}(\mathbf{Q}).
	\end{equation*}
	In order to solve this hyperbolic system numerically, we use here a classical second-order MUSCL-Hancock scheme \cite{van1984relation,Toro2009}, whose details are briefly recalled hereafter. As usual in the finite volume approach, we assume the discrete data $\mathbf{Q}_{i,j}^n$ to represent the cell averages over cell $\Omega_{i,j}$ at time $t^n$, i.e. 
	\begin{equation*}
		\mathbf{Q}_{i,j}^n = \frac{1}{\Delta x \Delta y} \int \limits_{\Omega_{i,j}} \mathbf{Q}(\mathbf{x},t^n) \, d \mathbf{x}.  
	\end{equation*}
	\begin{enumerate}
		\item \textbf{Reconstruction step:} In each cell $\Omega_{i,j}$, the left (L), right (R), bottom (B) and top (T) boundary extrapolated values are computed, respectively, as
		\begin{equation*}
			\mathbf{W}^L_{i,j} = \mathbf{Q}_{i,j}^n - \frac12 \partial_x \mathbf{Q}_{i,j}, \quad 
			\mathbf{W}^R_{i,j} = \mathbf{Q}_{i,j}^n + \frac12 \partial_x \mathbf{Q}_{i,j}, \quad 
			\mathbf{W}^B_{i,j} = \mathbf{Q}_{i,j}^n - \frac12 \partial_y \mathbf{Q}_{i,j}, \quad 
			\mathbf{W}^T_{i,j} = \mathbf{Q}_{i,j}^n + \frac12 \partial_y \mathbf{Q}_{i,j},
		\end{equation*}
		where the \textit{unlimited slopes} are given by 
		\begin{equation*}
			\partial_x \mathbf{Q}_{i,j} = \frac1{2\Delta x} \prn{\mathbf{Q}_{i+1,j}^n - \mathbf{Q}_{i-1,j}^n}, \quad 
			\partial_y \mathbf{Q}_{i,j} = \frac1{2\Delta y} \prn{\mathbf{Q}_{i,j+1}^n - \mathbf{Q}_{i,j-1}^n}, 
		\end{equation*}
		Since the original Cahn-Hilliard equation is dissipative, we assume all solutions to be smooth enough so that no nonlinear slope limiter needs to be employed, unlike what is usually proposed in the context of standard MUSCL-Hancock-type schemes \cite{Toro2009}.  
		\item \textbf{Half time-step update:} The boundary extrapolated values are evolved by a half-time step  as follows
		\begin{align*}
			\tilde{\mathbf{W}}^L_{i,j} = \mathbf{W}^L_{i,j}  +\frac{\Delta t}{2}\partial_t \mathbf{Q}_{i,j} , \quad 
			\tilde{\mathbf{W}}^R_{i,j} = \mathbf{W}^R_{i,j}  +\frac{\Delta t}{2}\partial_t \mathbf{Q}_{i,j}, \quad	
			\tilde{\mathbf{W}}^B_{i,j} = \mathbf{W}^B_{i,j}  +\frac{\Delta t}{2}\partial_t \mathbf{Q}_{i,j}, \quad 
			\tilde{\mathbf{W}}^T_{i,j} = \mathbf{W}^T_{i,j}  +\frac{\Delta t}{2}\partial_t \mathbf{Q}_{i,j},
		\end{align*}
		where the discrete time derivative is obtained via a Cauchy-Kowalewskaya procedure up to first order terms, i.e. 
		\begin{equation*}
			\partial_t \mathbf{Q}_{i,j} = -  \frac{1}{\Delta x} \prn{ \mathbf{F}\prn{\mathbf{W}^R_{i,j}} -\mathbf{F}\prn{\mathbf{W}^L_{i,j}}  } - \frac{1}{\Delta y} \prn{ \mathbf{G}\prn{\mathbf{W}^T_{i,j}} -\mathbf{G}\prn{\mathbf{W}^B_{i,j}}  } + \mathbf{S}(\mathbf{Q}_{i,j}^n).
		\end{equation*}
		\item \textbf{Intercell fluxes computation:} This depends on the choice of the approximate Riemann solver. For the Rusanov flux \cite{Rus62} we have
		\begin{align*}
			& \mathcal{F}_{i+\frac12,j} = \frac{1}{2}\prn{ \mathbf{F}(\tilde{\mathbf{W}}^R_{i,j}) +\mathbf{F}(\tilde{\mathbf{W}}^L_{i,j})} - \frac{1}{2} \lambda^{(x)}_{max}\prn{\tilde{\mathbf{W}}^R_{i,j} - \tilde{\mathbf{W}}^L_{i,j}}, \\
			& \mathcal{G}_{i,j+\frac12} = \frac{1}{2}\prn{ \mathbf{F}(\tilde{\mathbf{W}}^T_{i,j}) +\mathbf{F}(\tilde{\mathbf{W}}^B_{i,j})} - \frac{1}{2} \lambda^{(y)}_{max}\prn{\tilde{\mathbf{W}}^T_{i,j} - \tilde{\mathbf{W}}^B_{i,j}}, 
	\end{align*}
	where $\lambda_{max}$ is the maximum signal speed, obtained from the maximum eigenvalues as
	\begin{equation*}
		\lambda_{max}^{(x)} = \max_{k=1..7} \prn{\max  \prn{\left\vert\lambda^k_{i,j}\right\vert,\left\vert\lambda^k_{i+1,j}\right\vert} }, 	\quad 	\lambda_{max}^{(y)} = \max_{k=1..7} \prn{\max  \prn{\left\vert\lambda^k_{i,j}\right\vert,\left\vert\lambda^k_{i,j+1}\right\vert} }.
	\end{equation*}
	 Instead, for the FORCE flux \cite{Toro2009FORCE} we have 
		\begin{align*}
	 &\mathcal{F}_{i+\frac12,j} = \frac{1}{2}\prn{\mathbf{F} \prn{\frac12\prn{\tilde{\mathbf{W}}^R_{i,j}+\tilde{\mathbf{W}}^L_{i,j}}  + \frac{1}{2} \frac{\Delta t}{\Delta x} \prn{\mathbf{F}\prn{\tilde{\mathbf{W}}^R_{i,j}} - \mathbf{F}\prn{\tilde{\mathbf{W}}^L_{i,j}}}  }+\frac12\prn{\mathbf{F} \prn{\tilde{\mathbf{W}}^R_{i,j}}+\mathbf{F}\prn{\tilde{\mathbf{W}}^L_{i,j}}}  - \frac{1}{2} \frac{\Delta x}{\Delta t} \prn{\tilde{\mathbf{W}}^R_{i,j} - \tilde{\mathbf{W}}^L_{i,j}}  }, \\
	 &\mathcal{G}_{i,j+\frac12} = \frac{1}{2}\prn{\mathbf{G} \prn{\frac12\prn{\tilde{\mathbf{W}}^T_{i,j}+\tilde{\mathbf{W}}^B_{i,j}}  + \frac{1}{2} \frac{\Delta t}{\Delta y} \prn{\mathbf{G}\prn{\tilde{\mathbf{W}}^T_{i,j}} - \mathbf{G}\prn{\tilde{\mathbf{W}}^B_{i,j}}}  }+\frac12\prn{\mathbf{G} \prn{\tilde{\mathbf{W}}^T_{i,j}}+\mathbf{G}\prn{\tilde{\mathbf{W}}^B_{i,j}}}  - \frac{1}{2} \frac{\Delta y}{\Delta t} \prn{\tilde{\mathbf{W}}^T_{i,j} - \tilde{\mathbf{W}}^B_{i,j}}  },
		\end{align*}

	\item \textbf{Final explicit time update:} 
	\begin{equation*}
		\mathbf{Q}_{i,j}^{n+1} = \mathbf{Q}_{i,j}^{n} - \frac{\Delta t}{\Delta x} \prn{\mathcal{F}_{i+\frac12,j} - \mathcal{F}_{i-\frac12,j}} - \frac{\Delta t}{\Delta y} \prn{\mathcal{G}_{i,j+\frac12} - \mathcal{G}_{i,j-\frac12}} + \Delta t \,\mathbf{S} \prn{\mathbf{Q}_{i,j}^n + \frac{\Delta t}{2}\partial_t \mathbf{Q}_{i,j} }.
	\end{equation*}
		
	\end{enumerate}
	
	\section{Numerical results}
	\label{sec:results}
	In this section we present numerical results for some relevant benchmark problems.
	Concerning the hyperbolic approximation, in all what follows and unless explicitly stated otherwise, initial data is always imposed as follows
	\begin{equation}
		c(\x,0) = c_0(\x),\quad \varphi(\x,0) = c_0(\x),\quad  \p(\x,0) = \nabla c_0(\x), \quad w(\x,0) = \beta \Delta(c_0(\x)^3-c_0(\x)),\quad \q(\x,0) = 0.
		\label{eq:compatibleIC}
	\end{equation}

	\subsection{Convergence in $\alpha$ for a stationary solution} 
	In this test case, we would like to assess the accuracy of the presented hyperbolic approximation for a stationary state. In particular, we would like to study the influence of the penalty parameter $\alpha$ on the solutions of the hyperbolic system and to analyze at least numerically whether convergence can be observed in terms of the latter parameter. Thus, we propose to compare equivalent equilibrium solutions for both systems of equations \eqref{eq:CH1} and \eqref{eq:CH_Hyp} in one dimension of space. To do this, we consider both equations at equilibrium ($\partial/\partial t \equiv 0$) and we solve the corresponding Cauchy problem by choosing an initial point of the solution for both systems. As the equations are written in different variables, one must be careful into formulating the Cauchy problems in a way that would allow for a straightforward analogy between the two systems. Let us begin with the original system for which the Cauchy problem in one dimension of space is easily obtainable from \eqref{eq:CH_J} as follows 
	\begin{gather*}
		\pd{J}{x} = 0, \quad J = \gamma \pd{^3c}{x^3} - (3c^2-1)\pd{c}{x}, \\ J(x_0) = J^0, \quad c(x_0) = c^0, \quad \left.\pd{c}{x}\right\vert_{x=x_0} = c_I^0, \quad \left.\pd{^2c}{x^2}\right \vert_{x=x_0} = c_{I\!I}^0, 
	\end{gather*}    
	which can be written in a first-order system as 
	\begin{subequations}
		\begin{gather}
			c' = c_I,\quad c_I' = c_{I\!I}, \quad c_{I\!I}' = \frac{1}{\gamma}\prn{J +(3c^2-1)c_1 }, \quad  J' = 0, \\ c(x_0) = c^0, \quad c_I(x_0)= c_I^0, \quad c_{I\!I}(x_0) = c_{I\!I}^0, \quad J(x_0)= J^0, 
		\end{gather}   
		\label{eq:ODE1}
	\end{subequations}
	where \textit{prime} denotes differentiation with respect to $x$. The choices of these variables allow for an analogous Cauchy problem to be formulated for the hyperbolic system \eqref{eq:CH_Hyp} and which writes
	\begin{subequations}
		\begin{gather}
			\varphi' = p, \quad
			p' =  \frac{\alpha}{\gamma} \prn{\varphi-c}, \quad 
			c' =\frac{ \alpha  p - \tilde{q} }{3 c^2 -1 + \alpha}, \quad
			\tilde{q}' =0, \\
			\varphi(x_0) = c^0 + \frac{\gamma}{\alpha} c_{I\!I}^0, \quad  p(x_0) = c_I^0, \quad c(x_0) = c^0, \quad \tilde{q}(x_0) = - J^0, 
		\end{gather}
		\label{eq:ODE2}
	\end{subequations}
	where $\tilde{q} = q/\tau$. Note that since we are looking for stationary solutions, the Cauchy problem \eqref{eq:ODE2} is independent of $\beta$ and $\tau$, as they are both time-relaxation parameters. The initial-value problems given by equations \eqref{eq:ODE1} and \eqref{eq:ODE2} are both solved using a classical fourth-order Runge-Kutta scheme. The numerical computations are done over the one-dimensional domain $x\in[0,0.6]$ with a constant step size $\Delta x=10^{-5}$. The initial values are given by 
	\begin{equation*}
		c^0 = 1-10^{-6}, \quad  c^0_I = -10^{-5}, \quad c^0_{I\!I} = -10^{-10}, \quad J^0 = 10^{-8}.
	\end{equation*} 
	We take the free parameter $\gamma=10^{-4}$. First we solve the initial-value problem \eqref{eq:ODE1} from which we obtain a reference solution, which we denote hereafter by $\hat{c}(x)$. Then we run several simulations for the hyperbolic model with different values of the parameter $\alpha=\{25,100,400\}$. The comparison of these numerical solutions is given in Figure \ref{fig:conv_alpha}. A numerical convergence table is also provided in Table \ref{tab:alpha} for more values of $\alpha$. The obtained solution $c(x)$ for every $\alpha$ is compared with the reference solution $\hat{c}(x)$ as well as with $\phi(x)$ and the variable $p(x)$ is compared with $\hat{c}_1(x)$. We use here the discrete relative $L_2$ error defined as 
	\begin{equation*}
		\nrm{c-\hat{c}}_{L_2} = \frac{\sqrt{\sum_{i=1}^{N}  \, (c(x_i)-\hat{c}(x_i))^2}}{\sqrt{\sum_{i=1}^{N}  \, c(x_i)^2}}.
	\end{equation*}
	Figure \ref{fig:conv_alpha} shows that the solutions are in excellent agreement for sufficiently high values of the penalty parameter. When $\alpha$ is low, the figure shows that the wavelength of the solution is not correctly captured, thus generating a phase shift increasing with $x$. The amplitude and the overall shape of the solution seem to be well captured nevertheless. Table \ref{tab:alpha} shows that the numerical solution of system \eqref{eq:ODE2} converges towards that of system \eqref{eq:ODE1} with a first-order rate in terms of the penalty parameter $\alpha$.      
	\begin{figure}[H]
		\centering
		\includegraphics[]{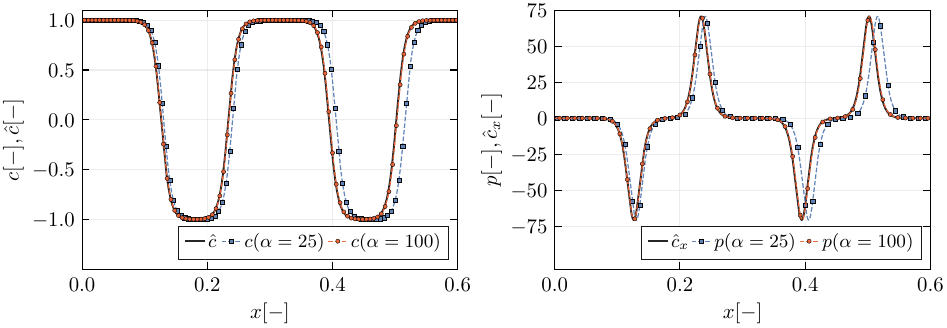}
		\caption{Comparison of a stationary solution of the hyperbolic Cahn-Hilliard model (discontinuous lines) with the original counterpart (solid line) for different values of the penalty parameter $\alpha$. The displayed solutions are obtained by solving the initial-value problems \eqref{eq:ODE1} and \eqref{eq:ODE2} on the domain $x\in[0,0.6]$ discretized with a constant step of size $\Delta x=10^{-5}$.}  
		\label{fig:conv_alpha}
	\end{figure}

	\begin{table}[H]
		\setlength{\tabcolsep}{12pt}
		\centering
		\small
		\begin{tabular}{lllllll}\toprule
			$\alpha$ & $\nrm{c-\hat{c}}_{L_2}$ & $\nrm{\p-\nabla \hat c}_{L_2}$ & $\nrm{c-\varphi}_{L_2}$  & $\mathcal{O}(c-\hat c)$ &$\mathcal{O}(\p - \nabla \hat c)$ & $\mathcal{O}(c-\varphi)$  \\
			
			\midrule
			$25$ &$2.64\times 10^{-1}$ & $5.66 \times 10^{-1}$ &$7.01 \times 10^{-3} $ & $-$ &$-$ & $-$ \\

			$50$ &$1.35\times10^{-1}$ & $3.02 \times 10^{-1}$  &$3.51 \times 10^{-3}$  &0.96 &0.90 &0.99  \\

			$100$ &$6.82\times 10^{-2}$ & $1.54 \times 10^{-1}$ &$1.75 \times 10^{-3}$ &0.99 &0.97 &0.99  \\

			$400$ &$1.70\times 10^{-2}$ & $3.86 \times 10^{-2}$ &$4.39 \times 10^{-4}$ &1.00 &0.99 &0.99  \\
			
			$1600$ &$3.80\times 10^{-3}$ & $8.64 \times 10^{-3}$ &$1.10 \times 10^{-4}$ &1.08 &1.08 &1.00  \\
			\toprule
		\end{tabular}
		\label{tab:alpha}
		\caption{Convergence table for the $L_2$ errors when comparing the numerical Cauchy problem solutions for the hyperbolic model with the the original Cahn-Hilliard equation. }
	\end{table}
	
	\subsection{Study of an exact solution}
	\label{exact}
	One can find a family of exact one-dimensional stationary periodic solutions to the Cahn-Hilliard system expressed as
	\begin{equation}
		c_{\epsilon}(x) = \sqrt{1-\epsilon}\ \mathrm{sn} \prn{\sqrt{\frac{\epsilon+1}{2\gamma}} (x-x_0),\ \sqrt{\frac{1-\epsilon}{1+\epsilon}}  }.
		\label{eq:Exact_SN}
	\end{equation}
	Here, $\mathrm{sn}(x,s)$ is the Jacobi elliptic sine function, whose second argument is chosen here as the elliptic modulus $s$. This solution family depends on a free real parameter $\epsilon\in[0,1]$ that determines its amplitude, period and elliptic modulus, and on $x_0$ only affecting its phase shift. It is worthy of note that in the limit $\epsilon\rightarrow 0$ corresponding to $s\rightarrow1$, one recovers the well-known solution 
	\begin{equation*}
		c(x) = \mathrm{tanh} \prn{\frac{x-x_0}{\sqrt{2\gamma}} }
	\end{equation*}
	as a particular case. Besides this particular limiting case, we recall that the Jacobi elliptic sine function is periodic for $0\leq s<1$ of period equal to $4K(s)$, where $K(s)$ is the complete elliptic integral of first kind given by the integral $K(s) = \int_0^{\pi/2} d \theta / \sqrt{1-s^2 \sin^2\theta}$. Thus, for fixed values of $\epsilon$ and $\gamma$, solution \eqref{eq:Exact_SN} is periodic in space and its wavelength is 
	\begin{equation*}
		\lambda = 4\sqrt{\frac{2\gamma}{\epsilon+1}} K\prn{\sqrt{\frac{1-\epsilon}{1+\epsilon}}}. 
	\end{equation*}
While \eqref{eq:Exact_SN} is an exact solution of the original Cahn-Hilliard equation \eqref{eq:CH1}, it is not a solution of the hyperbolic reformulation \eqref{eq:CH_Hyp}. This implies that the system will necessarily evolve in time, even if slightly and indiscernibly so.  Thus, it is of interest to see whether the numerical solution remains stable and to analyze some its features. For this test case, we set the values of $\epsilon=0.01$ and $\gamma=0.001$. Then, we consider a computational domain that expands over two wavelengths of the solution $[0,2\lambda]$, discretized over $N=2000$ cells and we impose periodic boundary conditions. We further take $\beta = 10^{-6}$, $\alpha=500$ and $\tau=8.10^{-4}$. The $\mathrm{CFL}$ number is taken equal to $0.95$ and the final simulation time is $t=10$. The comparison of the numerical solution with the corresponding initial data (that is the exact solution of Equation \eqref{eq:CH1}) is provided in Figure \ref{fig:exactSN}. We use here a MUSCL-Hancock scheme with the FORCE flux. 
	\begin{figure}[H]
		\centering
		\includegraphics[]{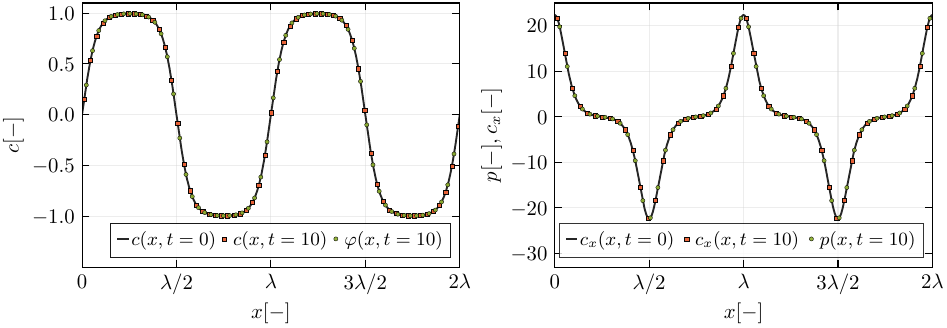}
		\caption{Comparison of the numerical solution obtained by solving numerically \eqref{eq:CH_Hyp} (shaped dots), with the equivalent exact solution of \eqref{eq:CH1}, which is provided as initial data (black line). The left graphic compares $c(x)$ (red squares) and $\varphi(x)$ (green circles) to the initial data. The right graphic compares the numerically evaluated $\partial c/\partial x$ (red squares) and $p(x)$ (green circles) to the gradient of the exact solution.   }  
		\label{fig:exactSN}
	\end{figure} 
The comparison shows perfect agreement between the exact solution of \eqref{eq:CH1} inserted as initial condition and the numerical solution shown at $t=10$, which is reached after approximately $15$M time-steps. This means that for the chosen parameters, the exact solutions of both PDE \eqref{eq:CH1} and \eqref{eq:CH_Hyp} are almost equal. 
In order to further analyze the behavior of the system and to also assess the efficiency of the numerical method, we provide three additional graphics in Figure \ref{fig:E_time}. The leftmost graphic shows the relative error of the total energy integral over the whole domain, with respect to its initial value, plotted over time for several values of the parameter $\tau$ and for the same mesh resolution of $N=2000$.
\begin{figure}[!t]
\centering
\includegraphics[]{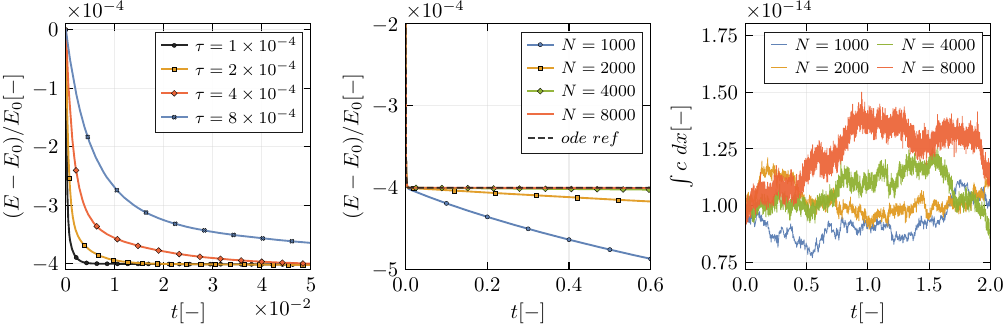}
\caption{Left: Relative error of the total energy over time for several values of $\tau$ and a fixed mesh resolution of $N=2000$. Center: Relative error of the total energy over time for a fixed value of $\tau = 10^{-4}$ and several mesh resolutions. The dashed lines correspond to the solution obtained by numerically integrating the ode \eqref{eq:Edecay_int}. Right: Numerical integral of $c$ over time for several mesh sizes (discrete mass conservation).}  
\label{fig:E_time}
\end{figure} 
One can see from the figure that energy decays towards a seemingly constant state, over a period of time that is longer for increasing values of $\tau$. The central graphic of Figure \ref{fig:E_time} shows the energy decay for same value of $\tau=10^{-4}$ but for several mesh resolutions. This allows to show that the observed dissipation in the leftmost graphic is rather due to the decay law \eqref{eq:Edecay_int} than to the numerical viscosity of the numerical method. Indeed the central graphic compares the evolution of the energy decay computed from the numerical solution to the one obtained by dynamically solving the ODE obtained by integrating \eqref{eq:Edecay_int}.
One can see that refining the mesh allows to obtain a better agreement with the latter. Lastly, the rightmost curve demonstrates that the numerical method conserves well the integral of $c$ over the computational domain, as the fluctuations are of the order of machine precision, i.e. we have verified discrete mass conservation.

\subsection{Spinodal decomposition}
The Cahn-Hilliard model is generally used to describe spinodal decomposition, \textit{i.e} a phase separation process in which the mixture of both phases becomes unstable and spontaneously leads to the formation of regions where only one of the phases is abundant. Usually, most test cases in the literature simulating this problem consider  \textit{random} initial data of low amplitude around $c=0$. Here, we deliberately propose a \textit{deterministic} expression in an attempt to guarantee reproducibility of the results. Thus, we suggest the following initial data
\begin{equation}
c(x) = 	\begin{cases}
		\phantom{-}0.01 \prn{\prn{\sin(10\pi (1+x)}-\sin\prn{10\pi (1+x)^2}}, \quad \mathrm{if} \  x\in[-1,0] \\ 
		-0.01 \prn{\prn{\sin(10\pi (1-x)}-\sin\prn{10\pi (1-x)^2}}, \quad \mathrm{if} \  x\in[0,1].
	\end{cases}
	\label{eq:spinodal}
\end{equation}
This function is built in such a way that it is $C^{\infty}$ over $[-1,1]$ as well as over $\mathbb{R}$ by periodic prolongation. This allows to consider periodic boundary conditions without having to worry about discontinuities on the boundaries. Given that the initial data \eqref{eq:spinodal} is a small amplitude perturbation of the unstable equilibrium $c=0$, the oscillations will grow in time forcing separation of the phases, until an equilibrium is eventually reached. We simulated this test using both the hyperbolic reformulation \eqref{eq:CH_Hyp}, as well as the original model \eqref{eq:CH1} for reference. We choose $\gamma=0.001$. For the hyperbolic approximation, we take $\beta = 10^{-7}$, $\alpha = 500$ and $\tau=10^{-5}$ and we take the computational domain $[-1,1]$, discretized over $N=2000$ cells. The CFL number is taken equal to $0.95$ and we use MUSCL-Hancock with FORCE flux. For the original model, the same domain is discretized over $N=1000$ points and we choose a fixed $\Delta t = 10^{-5}$. The final simulation time is $t=4$ for both simulations. In Figure \ref{fig:spinodal} a comparison of the numerical results obtained for both models is provided at several selected times, showcasing the solution dynamics.  
\begin{figure}[H]
	\centering
	\includegraphics[]{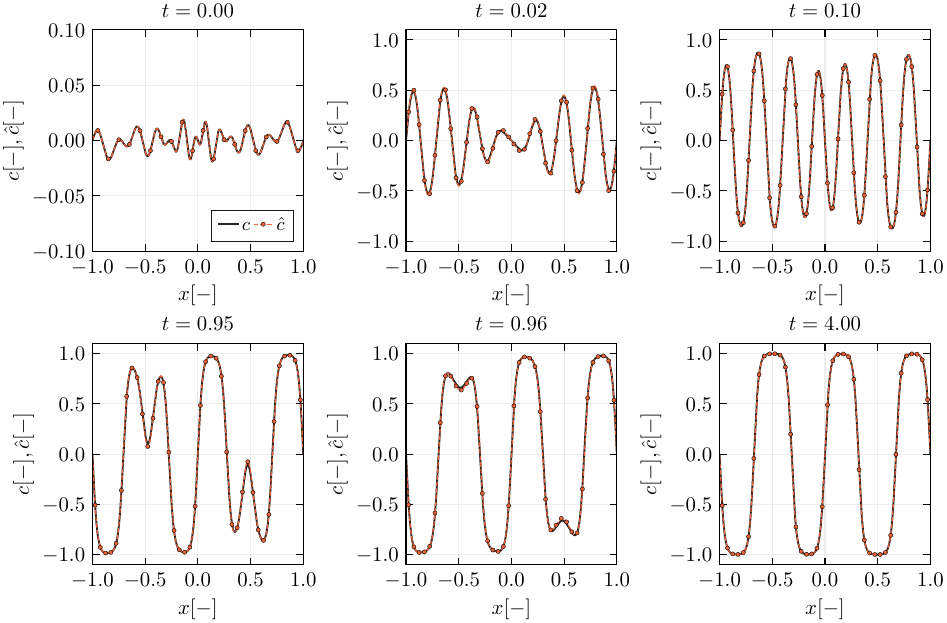}
	\caption{Comparison of the numerical results for the spinodal decomposition test case between the original Cahn-Hilliard model (orange dash-dotted lined) and its hyperbolic counterpart \eqref{eq:CH_Hyp} (solid black line), for several times.}  
	\label{fig:spinodal}
\end{figure} 
Both solutions show an excellent agreement between them and the dynamics are well-captured, especially in the earliest stage of the simulation where the growth dynamics are rather fast. We point out here that taking initial data as reported in \eqref{eq:compatibleIC} is important as soon as fast solution dynamics are involved. In fact, we show a sample of the same solution plotted for different initial data where well-preparedness is totally or partially absent. Thus, we reconsider the same configuration as above, and we only change the initial data, taken from the following set
\begin{alignat*}{5}
	IC^1: \quad c(x,0) &=c_0(x), \quad  &&\varphi(x,0) = c_0(x),\quad  &&\p(x,0) = 0, \quad && w(x,0) = 0,\quad &&\q(x,0) = 0, \\
	IC^2:\quad c(x,0) &=c_0(x), \quad &&\varphi(x,0) = c_0(x),\quad  &&\p(x,0) = \nabla c_0(x), \quad && w(x,0) = 0,\quad &&\q(x,0) = 0, \\
	IC^3:\quad c(x,0) &= c_0(x),\quad &&\varphi(x,0) = 0,\quad  &&\p(x,0) = 0, \quad &&w(x,0) = 0,\quad &&\q(x,0) = 0. \\
\end{alignat*}
For each of these unprepared initial conditions, we perform a numerical simulation which we then compare to the result of Figure \ref{fig:spinodal}, which was obtained from the well-prepared initial condition \eqref{eq:compatibleIC}. Figure \eqref{fig:wpIC} shows a comparison of the obtained numerical results from the initial conditions $IC^{1-3}$, denoted respectively by $c^{1-3}$ with the previous solution denoted now by $c^{wp}$, as well as $\hat{c}$, which is the reference numerical solution obtained from the original Cahn-Hilliard model.
One can see that significant differences between the solutions can be observed. The most drastic change is for $IC^1$, that is when $\p(\x,0)\neq \nabla\varphi(\x_0)$. Even though, one can see from the right graphic of Figure \ref{fig:wpIC} that $p^1$ and $c^1_x$ are still in good agreement, the solution had already been distorted in its earlier stages due to the initial data inconsistency and they are completely off the reference solution. 
\begin{figure}[!h]
	\centering
	\includegraphics[]{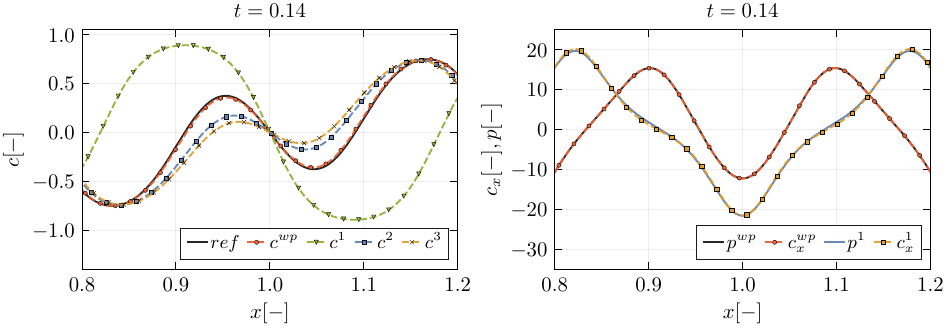}
	\caption{Left: Comparison of the numerical results for the initial conditions $IC^{1-3}$ (green triangles, blue squares and yellow squares, respectively) with the numerical simulation obtained with a well-prepared initial condition $c^{wp}$ (red circles) and $\hat{c}$ (solid black line). Right: Comparison of $p$ and $c_x$ for the numerical solutions $c^1$ and $c^{wp}$. Comparison is done at $t=0.14$ for the same parameters as \ref{fig:spinodal}.}  
	\label{fig:wpIC}
\end{figure}

\subsection{1D Ostwald Ripening}
Another test case for which the Cahn-Hilliard equation is well-known is the so-called Ostwald ripening phenomenon. Unlike spinodal decomposition, one starts here from already separated phases, however from a configuration where a phase is unequally distributed in the other. In this case, the smallest regions of the distributed phase diffuse and are absorbed into the larger regions. Thus, let us consider the following initial condition, adapted from \cite{neusser2015relaxation,hitz2020parabolic,dhaouadi2022NSK}, and which consists in two neighboring different sized 1D "bubbles" of the type    
\begin{equation*}
	c(x,0) = 1 + \sum_{i=1}^2 \tanh\prn{\frac{x-x_i-r_i}{\sqrt{2\gamma}}} - \tanh\prn{\frac{x-x_i+r_i}{\sqrt{2\gamma}}},
\end{equation*}
where $x_1=0.30$, $x_2=0.75$ and $r_1=0.12$, $r_2=0.06$ are the positions and the radii of the two bubbles, respectively. For this simulation we take $\gamma=10^{-3}$. For the hyperbolic model we also take $\alpha=1000$, $\tau=10^{-4}$ and $\beta=10^{-7}$. The computational domain is $[0,1]$ discretized over $N=1000$ cells for both hyperbolic and original models. The CFL number is taken equal to $0.95$ for the hyperbolic model and we fix $\Delta t=10^{-4}$ for the original one. The final simulation time is $t=0.3$. A comparison of both results showcasing the main dynamics of the solution are presented in Figure \ref{fig:OR1D}.
\begin{figure}[!b]
	\centering
	\includegraphics[]{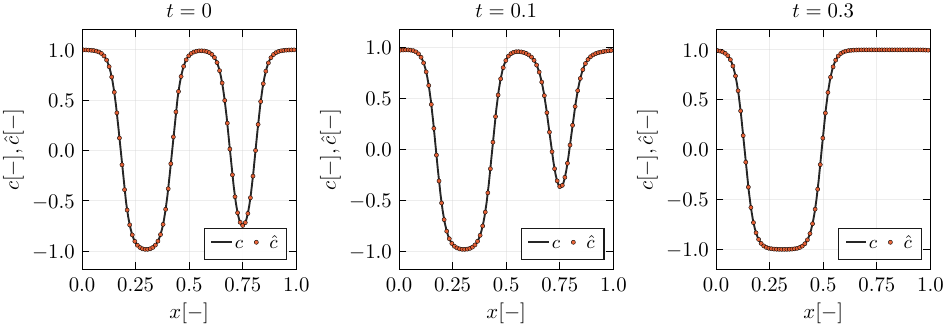}
	\caption{Comaprison of the numerical solutions for hyperbolic Cahn-Hilliard model \eqref{eq:CH_Hyp} (black line) and the original model \eqref{eq:CH1} (red dots)for the Ostwald Ripening test case at times $t=\{0,0.1,0.3\}$. }  
	\label{fig:OR1D}
\end{figure} 
The comparison of both solutions shows an excellent agreement, as the smaller bubble diffuses and gets absorbed by the larger one, which then becomes stationary. The chosen hyperbolic model parameters are such that no observable delay in the solution dynamics is present.

\subsection{2D stationary radial solution}
In order to obtain a radially symmetric equilibrium solution, we consider here a pseudo-transient continuation method, \textit{i.e.} we numerically solve the time-dependent Cahn-Hilliard equation \eqref{eq:CH1} in polar coordinates $r-\theta$, assuming radial symmetry, i.e. $\partial / \partial \theta = 0$. We thus get 
\begin{equation*}
\pdt{c} - \frac{1}{r}\pd{}{r}\prn{r \pd{}{r}\prn{c^3 - c - \frac{\gamma}{r} \pd{}{r}\prn{r \pd{c}{r}}}} = 0, \qquad r = \sqrt{x^2+y^2},
\end{equation*}
and we consider the following initial condition 
\begin{equation*}
	c_0(r) = -\tanh\prn{\frac{r-0.5}{\sqrt{2\gamma}}},
\end{equation*}  
which is then evolved towards a stationary state. We use a semi-implicit conservative finite difference scheme on a very fine radial mesh to obtain a reliable reference solution. In the following the value of $\gamma=10^{-3}$ is taken.  This method offers more flexibility than solving a singular Cauchy problem for the corresponding stationary ODE, as the prescription of initial data in the vicinity of $r=0$ offers little control of the desired shape of the solution and its behavior. After convergence towards equilibrium, the obtained solution is then inserted as initial data for the hyperbolic model, in order to verify if it remains stationary. Conversion of the data from 1D radial coordinate to the 2D Cartesian grid is done using classical high order piecewise polynomial Lagrange interpolation of degree four, allowing to obtain the values of $c(x,y)$ and $\p(x,y)$. 
\begin{figure}[H]
	\centering
	\includegraphics[]{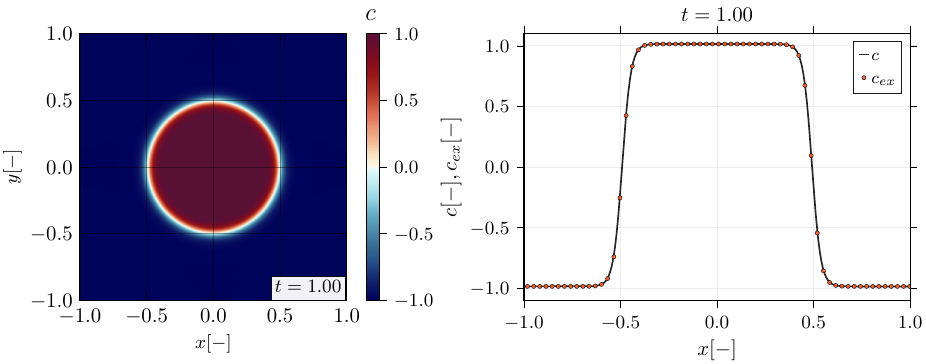}
	\caption{Left: 2D color plot of the radial stationary solution. Right: radial cut of the numerical solution along the line $y=0$ at $t=1$ (black line), compared with the exact solution $c_{ex}$, provided as initial data (red dots). }  
	\label{fig:Bubble}
\end{figure} 
We consider for this simulation a square domain $[-1,1]\times[-1,1]$, discretized over $500\times500$ cells. We take $\gamma=10^{-3}$, $\alpha = 500$, $\beta=10^{-6}$ and $\tau=10^{-4}$. The CFL number is set to $0.9$. The numerical solution is shown in Figure \ref{fig:Bubble}, together with a comparison with the radial reference solution. The numerical solution of the hyperbolic system, plotted at $t=1.00$ is visibly the same as the radial reference solution of the original Cahn-Hilliard equation. The final simulation time was reached after nearly $15$M time-steps, meaning that the stationary equilibrium is preserved also in a long-term regime.   

\subsection{2D Ostwald ripening}
We finally consider the process of Ostwald ripening in two space dimensions and we take an initial condition given by the expression 
\begin{equation}
	c(x,y,0) = 1 + \sum_{i=1}^8 \tanh\prn{\frac{r_i(x,y)-r^b_{i}}{\sqrt{2\gamma}}} - \tanh\prn{\frac{r_i(x,y)+r^b_{i}}{\sqrt{2\gamma}}}, \qquad r_i(x,y) = \sqrt{(x-x_i)^2 + (y-y_i)^2},
	\label{eq:IC_OR2D}
\end{equation} 
where $(x_i,y_i)$ are the Cartesian coordinates of the $i^{th}$ bubble, $r_i$ is the corresponding local radial coordinate and $r^b_i$ is its approximate initial radius. The corresponding values are summarized in Table \ref{tab:radii}.
\begin{table}[H]
	\centering
	\begin{tabular}{l|rrrrrrrr}\toprule
		$i$ & $1$ & $2$ &$3$& $4$ & $5$ &$6$ &$7$ & $8$
		\\ 
		\midrule
		$x_i$   & $0.00$ & $-0.30$ & $-0.30$ & $-0.35$ & $0.00$& $0.25$ & $0.30$  & $0.35$\\
		$y_i$   & $0.10$ & $-0.40$ & $0.40$  & $0.00$ & $-0.30$& $0.45$ & $-0.35$ & $0.05$\\
		$r^b_i$ & $0.15$ &  $0.10$ & $0.10$  & $0.06$ & $0.07$ & $0.06$ & $0.08$ & $0.07$\\ 
		\bottomrule
	\end{tabular}
	\caption{Coordinates and radii for the initial condition \eqref{eq:IC_OR2D}, for the 2D Ostwald ripening test case.}
	\label{tab:radii}
\end{table}
For this test case we take $\gamma=10^{-3}$. The computational domain is $[-0.5,0.5]\times[-0.6,0.6]$. For the hyperbolic model, we take $\alpha=1000$, $\beta=10^{-8}$ and $\tau=10^{-5}$, we consider a discretization over $600\times720$ Cartesian square cells and the CFL number is set to $0.9$. For the original model, we discretize the domain similarly and we take $\Delta t=10^{-5}$. Periodic boundary conditions in both directions are considered for both simulations and final time is $t=1$. 
\begin{figure}[H]
	\centering
	\includegraphics[]{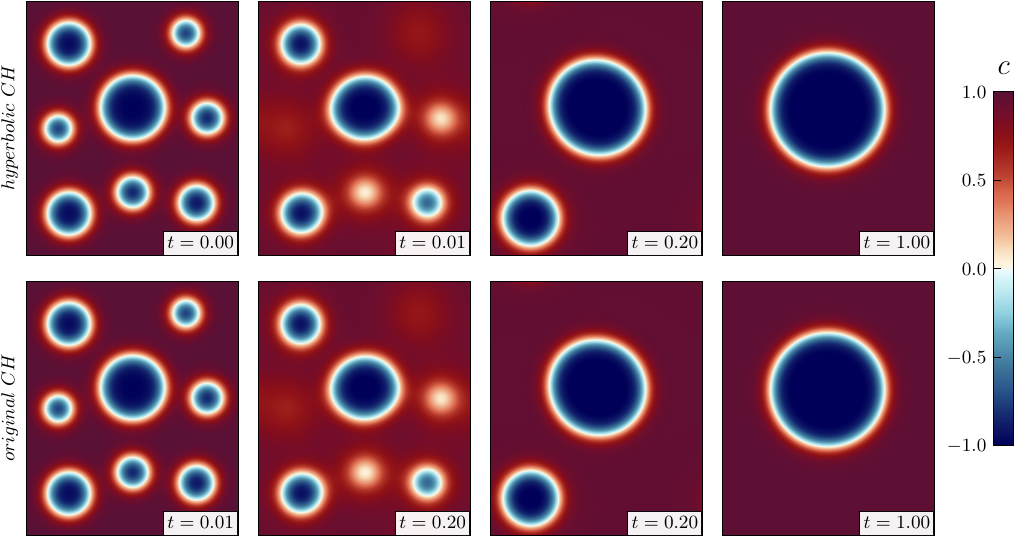}
	\caption{Comparison of the numerical solutions for hyperbolic Cahn-Hilliard model \eqref{eq:CH_Hyp} (Top) with the reference numerical solution obtained for the original model \eqref{eq:CH1} (Bottom) for the 2D Ostwald Ripening test case at times $t=\{0.00,0.01,0.20,1.00\}$. }  
	\label{fig:OR2D}
\end{figure} 
A comparison of the numerical results for both models, at several times, is depicted in Figure \ref{fig:OR2D}. The figure shows excellent qualitative agreement between both numerical solutions, for the intermediate stages as well as for the reached equilibrium state. The dynamics of the solutions are well-captured at the correct times and the 2D shape of the dispersed phase regions comply with the reference solution. In order to further quantify the quality of the approximation, we also provide 1D plots over two horizontal cuts of the solution at $y=0$ and $y=0.4$ in figure \ref{fig:OR2D_cut}. 
The comparison shows that the obtained numerical solution for the augmented hyperbolic relaxation system \eqref{eq:CH_Hyp} matches  perfectly with the reference numerical solution obtained for the original model \eqref{eq:CH1} also quantitatively for the selected times, which best showcase the main solution dynamics. 

\begin{figure}[H]
	\centering
	\includegraphics[width=\textwidth]{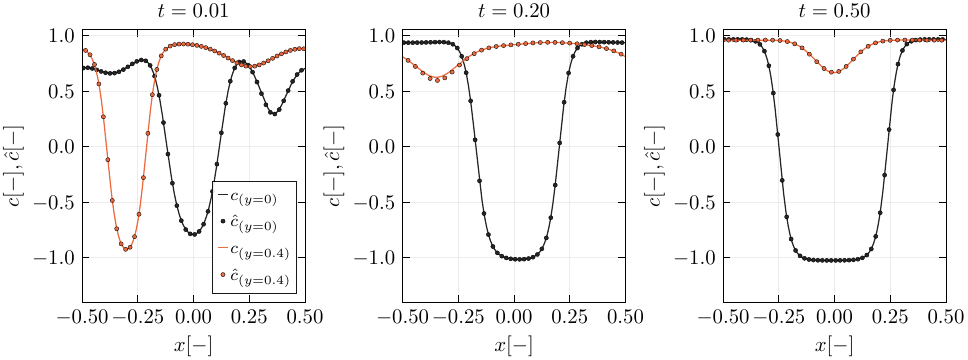}
	\caption{Comparison of the numerical solutions for hyperbolic Cahn-Hilliard model (continuous lines) with the reference numerical solution obtained for the original model \eqref{eq:CH1} (dots) for the 2D Ostwald Ripening test case at times $t=\{0.01,0.20,0.50\}$ over the lines $y=0$ (red) and $y=0.4$ (black).}  
	\label{fig:OR2D_cut}
\end{figure}

\section{Conclusion}
\label{sec:conclusion}
In this work, we have presented a new first order hyperbolic relaxation model that is able to successfully approximate the Cahn-Hilliard equation, which contains spatial derivatives of order up to four. The former was shown to by hyperbolic, to admit a Lyapunov functional and to conserve the total mass, as for the original equation.
A novel scheme is also proposed to solve the original Cahn-Hilliard equation based on semi-implicit conservative finite differences in one and two dimensions of space. 
The numerical results for both the proposed hyperbolic model and the original one, show excellent agreement, both qualitatively and quantitatively, in one and multiple space dimensions, even in the presence of fast dynamics and strong topological changes. 
A natural extension of this work would be to extend it to the description of phase transitions in continua \cite{lowengrub1998quasi,mulet2024implicit}. Additional optimization of the numerical method for the hyperbolic model could further improve its performance, for instance through a semi-implicit discretization. 

\section*{Acknowledgements}
This research was funded by the Italian Ministry of Education, University and Research (MIUR) in the frame of the Departments of Excellence Initiative 2018--2027 attributed to DICAM of the University of Trento (grant L. 232/2016) and via the PRIN 2022 project \textit{High order structure-preserving semi-implicit schemes for hyperbolic equations}.  
FD was also funded by NextGenerationEU, Azione 247 MUR Young Researchers – SoE line. 
MD was also funded by the European Union NextGenerationEU project PNRR Spoke 7 CN HPC.  
Views and opinions expressed are however those of the author(s) only and
do not necessarily reflect those of the European Union or the European Research
Council. Neither the European Union nor the granting authority can be held
responsible for them. 
FD and MD are members of the Gruppo Nazionale Calcolo Scientifico-Istituto Nazionale di Alta Matematica (GNCS-INdAM). FD would like to acknowledge support from the CINECA under the ISCRA initiative, for the availability of high-performance computing resources and support (project number IsB27\_NeMesiS). SG would like to thank Alain Miranville for helpful discussions.

\bibliographystyle{mystyle}
\bibliography{./references}

\end{document}